\documentclass[12pt,reqno]{amsart}
\usepackage{geometry}                % See geometry.pdf to learn the layout options. There are lots.
\usepackage{graphicx}
\usepackage{amssymb}
\usepackage{aliascnt}
\usepackage{doi}
\usepackage{epstopdf}
\usepackage{tikz}
\usepackage{pgfplots}
\usepgfplotslibrary{fillbetween}
\usepackage{xcolor}

\usepackage{esint}

\newtheorem{theorem}{Theorem}[section]

\newaliascnt{corx}{thmx}

\aliascntresetthe{corx}

\newaliascnt{lemma}{theorem}
\newtheorem{lemma}[lemma]{Lemma}
\aliascntresetthe{lemma}

\newaliascnt{proposition}{theorem}
\newtheorem{proposition}[proposition]{Proposition}
\aliascntresetthe{proposition}

\newaliascnt{corollary}{theorem}
\newtheorem{corollary}[corollary]{Corollary}
\aliascntresetthe{corollary}

\newaliascnt{conjecture}{theorem}
\newtheorem{conjecture}[conjecture]{Conjecture}
\aliascntresetthe{conjecture}

\newaliascnt{example}{theorem}

\aliascntresetthe{example}

\newaliascnt{question}{theorem}

\aliascntresetthe{question}

\newcommand{\B}{{\mathbb B}}
\newcommand{\C}{{\mathbb C}}
\newcommand{\deltastar}{\Delta^{\! \s}} 
\newcommand{\e}{\varepsilon}
\newcommand{\R}{{\mathbb R}}
\newcommand{\Rm}{{{\mathbb R}^m}}
\newcommand{\Rn}{{{\mathbb R}^n}}
\newcommand{\Rnp}{{{\mathbb R}^{n+1}}}
\newcommand{\Sph}{{\mathbb S}}
\newcommand{\Sphn}{{{\mathbb S}^n}}
\newcommand{\s}{{\text{\small $\star$}}}
\newcommand{\xh}{\hat{x}}

\newcommand{\Real}{\operatorname{Re}}

\newcommand{\capone}{\operatorname{Cap_1}}
\newcommand{\capp}{\operatorname{Cap_\mathit{p}}}
\newcommand{\capnone}{\operatorname{Cap_{\mathit{n}-1}}}
\newcommand{\capmtwo}{\operatorname{Cap_{\mathit{m}-2}}}
\newcommand{\capntwo}{\operatorname{Cap_{\mathit{n}-2}}}
\newcommand{\capzero}{\operatorname{Cap_0}}

\title[Moments of equilibrium measures]{Balls minimize moments of logarithmic and Newtonian equilibrium measures}
\author{Carrie Clark and Richard S. Laugesen}
\email{carriec2@illinois.edu, Laugesen@illinois.edu}
\address{University of Illinois, Urbana, IL 61801, USA}

\keywords{Riesz capacity, electrostatic, potential theory}
\subjclass[2020]{\text{Primary 31A15, 31B15. Secondary 35B51}}

\begin{document}

\begin{abstract}
The $q$-th moment ($q>0$) of electrostatic equilibrium measure is shown to be minimal for a centered ball among $3$-dimensional sets of given capacity, while among $2$-dimensional sets a centered disk is the minimizer for $0<q \leq 2$. Analogous results are developed for Newtonian capacity in higher dimensions and logarithmic capacity in $2$ dimensions. Open problems are raised for Riesz equilibrium moments. 
\end{abstract}

\maketitle

\section{\bf Introduction}

The moments of electrostatic equilibrium measure on a conductor quantify how much that conductor ``spreads out''. It seems plausible that among sets with given capacity, the ball should minimize moments of equilibrium measure. This paper proves minimality of certain moments for the ball among $n$-dimensional sets with given Newtonian capacity in $n$ and $n+1$ dimensions. These results suggest a conjecture for a whole family of Riesz capacities. 

Riesz capacity $\capp(\cdot)$ and logarithmic capacity $\capzero(\cdot)$, along with the corresponding energies $V_p$ and $V_{log}$ and their equilibrium measures, are defined in \autoref{sec:capacities}. The notion of inner capacity zero is defined there too. The Newtonian case occurs when $p=n-2$ and $n \geq 3$. 

\subsection{Minimal moments of equilibrium measures} The first theorem says that among compact sets with given Newtonian capacity in $\Rn, n \geq 3$, or with given logarithmic capacity in $\R^2$, the $q$-th moment of equilibrium measure with $q>0$ is minimal for the ball. 

%In the Newtonian case this condition means $\int_K |x|^{-(n-2)} \, d\mu = V_{n-2}(K)$. In the logarithmic case it means $\int_K \log 1/|x| \, d\mu = V_{log}(K)$. In particular, if the origin is an interior point of $K$ then it is a regular point. See \autoref{sec:background} for this and other background facts from potential theory. 
%
 \begin{theorem}[Moments of Newtonian equilibrium measure are minimal for the ball]  \label{th:momentnewton}
Let $n \geq 2$. Suppose the compact set $K \subset \Rn$ has the same $(n-2)$-capacity as a closed ball $B \subset \Rn$ centered at the origin, $\capntwo(K)=\capntwo(B)>0$, and write $\mu$ and $\nu$ for the $(n-2)$-equilibrium measures of $K$ and $B$ respectively. 

(i) If $q>0$, or else if $q<-(n-2)$ and the origin is a regular point of $K$, then 
\[
\int_K |x|^q \, d\mu \geq \int_B |x|^q \, d\nu .
\]

(ii) Letting $q \to 0$, 
\[
\int_K \log |x| \, d\mu \geq \int_B \log |x| \, d\nu .
\]

(iii) If $-(n-2) \leq q < 0$ then the inequality is reversed:
\[
\int_K |x|^q \, d\mu \leq \int_B |x|^q \, d\nu .
\]

(iv) (Equality statements) In part (i), equality holds if and only if $K$ contains the sphere $\partial B$ and $K \setminus B$ has inner $(n-2)$-capacity zero; the same is true for part (ii) when $n \geq 3$, and for part (iii) when $-(n-2) < q < 0$. For part (ii) with $n=2$, and part (iii) with $n \geq 3$ and $q=-(n-2)$, equality holds if and only if the origin is a regular point of $K$.
\end{theorem}
A \emph{regular point} of $K$ is one at which the equilibrium potential equals the energy of the set, which indeed is the case at every point of $K$ except for a subset of inner capacity zero.

In part (i) for negative exponents $q<-(n-2)$, the moment $\int_K |x|^q \, d\mu$ on the left side of the inequality might equal $+\infty$ for some $K$. Also, that moment can be arbitrarily close to $0$ if the hypothesis that the origin belongs to $K$ is dropped, since then $K$ could lie far from the origin. On the other hand, the inequality in part (iii) goes in the reverse direction and so holds trivially when $K$ lies far from the origin and the left side is close to $0$. 

We prove the theorem in \autoref{sec:momentnewtonproof} by a straightforward application of Jensen's inequality. The proof yields stronger inequalities than are stated in the theorem, giving moment inequalities for $F(1/|x|^{n-2})$ when $F$ is convex decreasing. The crucial fact is that Newtonian equilibrium measure $\nu$ on a ball in $\Rn$ is simply normalized surface area measure on the boundary sphere, which means also that the lower bound in \autoref{th:momentnewton} evaluates to $\int_B |x|^q \, d\nu = R^q$, where $R$ is the radius of the ball $B$.

More difficult is the question of whether the ball minimizes equilibrium moments beyond the Newtonian case, that is, when the Riesz parameter is greater than $n-2$ and the equilibrium measure of the ball is not concentrated on the boundary sphere but instead distributes charges throughout the interior. The main result of the paper proves such moment inequalities when $p=n-1$.  
\begin{theorem}[Moments of $(n-1)$-equilibrium measure are minimal for the ball]  \label{th:momentnone}
Let $n \geq 1$ and suppose the compact set $K \subset \Rn$ has the same $(n-1)$-capacity as a closed ball $B \subset \Rn$ centered at the origin, meaning $\capnone(K)=\capnone(B)>0$. Then
\[
\int_K |x|^q \, d\mu \geq \int_B |x|^q \, d\nu , \qquad q \in (0,2] , 
\]
where $\mu$ and $\nu$ are the $(n-1)$-equilibrium measures of $K$ and $B$, respectively. Equality holds if and only if $K=B \cup Z$ where $Z$ has inner $(n-1)$-capacity zero. 

Further, the logarithmic moments satisfy 
\[
\int_K \log |x| \, d\mu \geq \int_B \log |x| \, d\nu .
\]
When $n \geq 2$, equality holds for the logarithmic moments if and only if $K=B \cup Z$ for some $Z$ with inner $(n-1)$-capacity zero. When $n=1$, equality holds if and only if the origin is a regular point of $K$ (meaning $\int_K \log 1/|x| \, d\mu = V_{log}(K)$). 
\end{theorem}
The theorem is proved in \autoref{sec:J}--\autoref{sec:momentnoneproof} by suitably adapting Baernstein's $\star$-function method from complex analysis. The set-up for the proof is shown in \autoref{fig:K}. Note that the lower bounds in the theorem can be computed explicitly, by means of the known formula (\autoref{sec:background}) for the equilibrium measure $\nu$ of the ball.  

The Riesz capacity in this theorem with $p=n-1$ is simply Newtonian capacity for $\Rnp$, since the Newtonian case arises when $p$ equals the dimension minus $2$. Hence \autoref{th:momentnone} can be interpreted as saying that among compact sets in $\Rn$ having given Newtonian capacity when regarded as subsets of $\Rnp$, the $q$-th moment of equilibrium measure is minimal for the $n$-ball, when $0<q \leq 2$. 

\begin{figure}
\begin{center}
\begin{tikzpicture}[x=0.5cm,y=0.5cm,z=0.3cm,>=stealth,
%\declarefunction{F(\t)=(1+0.08*cos(6*\t)+0.05*sin(8*\t)+0.08*cos(9*\t))*(5*cos(\t))},
%\declarefunction{G(\t)=(1+0.08*cos(6*\t)+0.05*sin(8*\t)+0.08*cos(9*\t))*(sin(\t))}
]
%the set K - 
 \filldraw[gray!60] plot[smooth,samples=36,domain=0:360,variable=\t] ({1.1*(1+0.08*cos(6*\t)+0.05*sin(8*\t)+0.08*cos(9*\t))*(4.5*cos(\t))},{1.1*(1+0.08*cos(6*\t)+0.05*sin(8*\t)+0.08*cos(9*\t))*(sin(\t))});
 \draw plot[smooth,samples=36,domain=0:360,variable=\t] ({1.1*(1+0.08*cos(6*\t)+0.05*sin(8*\t)+0.08*cos(9*\t))*(4.5*cos(\t))},{1.1*(1+0.08*cos(6*\t)+0.05*sin(8*\t)+0.08*cos(9*\t))*(sin(\t))});
 \draw (3,-1.4) node[right] {$K$};
 %point x in K
\filldraw  (1.9,0.5) circle (1.5pt) node[right] {$x$};
% The axes
\draw[-] (xyz cs:x=-6) -- (xyz cs:x=6.1) node[right] {$\Rn$};
\draw[-] (xyz cs:y=-4) -- (xyz cs:y=6) node[above] {$\R$};
\draw[-] (xyz cs:z=-5.25) -- (xyz cs:z=4.25);
%ellipse for the ball in the horizontal slice
\draw (0,4.4) ellipse (1.7cm and 0.4cm);
%label for z
\draw[-] (-0.1,4.4) -- (0.2,4.4) ;
\draw (0,4.4)  node[left] {$z$};
%radius line and label
\draw[-] (0,4.4) -- (3.37,4.4);
\draw (1.5,4.4-0.1) node[above] {$r$};
\end{tikzpicture}
\end{center}
\caption{The compact set $K$ in \autoref{th:momentnone} lies in $\Rn$. For the proof, we regard $K$ as lying in the higher dimensional space $\Rnp$ and proceed by integrating the (harmonic) equilibrium potential over the $n$-dimensional ball at height $z$ with radius $r$. 
\label{fig:K}
}
\end{figure}
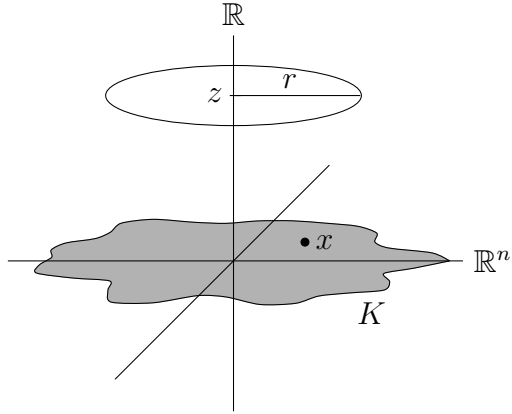

This interpretation is particularly informative when $n=1$ and $n=2$. 
\begin{corollary}[Moments of planar sets]  \label{co:momentplanar} Let $0 < q \leq 2$. 

If a compact set $K \subset \R^2$ has the same Newtonian capacity as a closed disk $B$ centered at the origin, $\capone(K)=\capone(B)$, then $\int_K |x|^q \, d\mu \geq \int_B |x|^q \, d\nu$ where $\mu$ and $\nu$ are the Newtonian equilibrium measures of $K$ and $B$ regarded as sets in $\R^3$. 

If a compact set $K \subset \R$ has the same logarithmic capacity as a closed interval $I$ centered at the origin, $\capzero(K)=\capzero(I)$, then $\int_K |x|^q \, d\mu \geq \int_I |x|^q \, d\nu$ where $\mu$ and $\nu$ are the logarithmic equilibrium measures of $K$ and $I$ regarded as sets in $\R^2$. 
\end{corollary}

Does the last theorem remain true for a whole range of Riesz capacities and moments? 
\begin{conjecture}[Moments of Riesz $p$-equilibrium measure are minimal for the ball] \label{co:moments}
Suppose $n \geq 2$ and $n-2 < p < n$, or else $n=1$ and $0 \leq p < 1$. If a compact set $K \subset \Rn$ and closed ball $B \subset \Rn$ centered at the origin have the same $p$-capacity, $\capp(K)=\capp(B)$, then
\[
\int_K |x|^q \, d\mu \geq \int_B |x|^q \, d\nu , \qquad q \in (0,\infty), 
\]
where $\mu$ and $\nu$ are the $p$-equilibrium measures of $K$ and $B$, respectively. 
\end{conjecture}
This conjecture is true in the limit $q \to \infty$, since if one takes the $q$-th root of the moment inequality and lets $q \to \infty$ then it says $\lVert x \rVert_{L^\infty(\mu)} \geq \lVert x \rVert_{L^\infty(\nu)}$, which must hold because otherwise $\mu$ would be supported in a ball strictly contained in $B$, contradicting the assumption that $K$ and $B$ have the same capacity. 

The case $p=n-2$ of the conjecture was handled in \autoref{th:momentnewton}, while \autoref{th:momentnone} covers $p=n-1$ with $0 < q \leq 2$. 

When $p<n-2$, the conjecture might still be plausible. Our proof in \autoref{sec:momentnewtonproof} for $p=n-2$ does not extend to $p<n-2$ because the equilibrium potential of $K$ at the origin can now have value greater than the energy $V_p(K)$ (see \cite[p.{\,}163]{L72}) and so inequality \eqref{eq:potentialineq} in the proof can fail. 

However, the fact that the equilibrium measure of the ball is supported on the boundary when $p<n-2$ can be used to show that if the moment inequality is true for some moment, then it holds also for all higher moments. 
\begin{proposition}[Range of moments for $p$-equilibrium measures with $0<p<n-2$]\label{pr:momentsthreshold}
Let $n \geq 3$ and $0<p<n-2$. Suppose the compact set $K \subset \Rn$ and closed ball $B \subset \Rn$ centered at the origin have the same $p$-capacity, $\capp(K)=\capp(B)>0$. If the equilibrium moment inequality
\begin{equation} \label{eq:momentsthreshold}
\int_K |x|^q \, d\mu \geq \int_B |x|^q \, d\nu
\end{equation}
holds for some $q=q_* \in (0,\infty)$ then it holds for all $q  \in [ q_*,\infty)$.
\end{proposition}
\autoref{sec:threshold} has the proof. Some such threshold $q_*=q_*(K,n,p)$ exists, since the moment inequality is true in the limit $q\to \infty$, as observed above. What one would like to show is the existence of a threshold $q_*$ that does not depend on the set $K$.

\subsection{\bf Definitions of energy, capacity, and equilibrium measure}
\label{sec:capacities}

References for the following definitions and facts can be found in \autoref{sec:background}. 

Consider a compact subset $K$ in $\R^n, n \geq 1$. 
The \textbf{logarithmic energy} of $K$ is
\[
V_{log}(K) = \inf_\mu \int_K \! \int_K \log \frac{1}{|x-y|} \, d\mu(x) d\mu(y)
\]
where the infimum is taken over all probability measures on $K$, that is, positive unit Borel measures. (Such measures always exist unless $K$ is empty, and for the empty set we define $V_{log}(\emptyset)=+\infty$.) The energy is greater than $-\infty$ since $|x-y|$ is bounded. The infimum is attained, and if the energy is less than $+\infty$ then it is attained by a unique minimizing measure called the logarithmic equilibrium measure. 

The \textbf{Riesz $p$-energy} of $K$ is
\[
V_p(K) = 
\inf_\mu \int_K \! \int_K \frac{1}{|x-y|^p} \, d\mu(x) d\mu(y) , \qquad p>0 ,
\]
where  the infimum is taken over all probability measures on $K$. Again the infimum is attained, and if the energy is finite then the minimizing measure is unique and we call it the $p$-equilibrium measure. The \textbf{Riesz $p$-capacity} is 
\[
\capp(K) = 
\begin{cases}
V_p(K)^{-1/p} , & p > 0 , \\
\exp(-V_{log}(K)) , & p = 0 .
\end{cases}
\]
This capacity is positive if and only if the energy is finite. The $0$-capacity is also known as \textbf{logarithmic capacity}, and the logarithmic equilibrium measure will be called the $0$-equilibrium measure.  

We say a set $Z$ has \emph{inner $p$-capacity zero} if $\capp(Z^\prime)=0$ for every compact $Z^\prime \subset Z$. 

\textbf{Newtonian energy} is  $V_{n-2}(K)$ and \textbf{Newtonian capacity} is $\capntwo(K)$. That is, the Newtonian case arises when 
\[
n \geq 3 \quad \text{and} \quad p=n-2. 
\]
Capacity formulas for the ball are collected in \autoref{sec:background}. 

Note that capacity is monotonic with respect to set inclusion, with 
\[
K_1 \subset K_2 \quad \Longrightarrow \quad \capp(K_1) \leq \capp(K_2) ,
\]
because the larger set supports a larger collection of probability measures and hence has smaller energy. 

Some authors, such as Landkof \cite{L72} or Borodachov, Hardin and Saff \cite{BHS19}, define $\capp(K)$ to be $1/V_p(K)$ without taking the $p$-th root. The current definition is more convenient for our purposes since it ensures that capacity scales linearly: 
\[
\capp(sK)=s \capp(K) , \qquad s>0 .
\]

\subsection{Prior work on moments of equilibrium measure}

Only a few extremal results on moments of equilibrium measure are known to us in the literature, as summarized below. None overlap directly with the results in this paper. 

For logarithmic equilibrium measure in the plane, Baernstein, Laugesen and Pritsker \cite[Theorem 1]{BLP11} prove an inequality somewhat related our to \autoref{th:momentnone} when $n=1$, for compact $K \subset \R$. They show that if $K$ is translated to put its electrostatic centroid at the origin ($\int_K x \, d\mu=0$) and $I$ is a centered interval with the same logarithmic capacity as $K$, then the lower bound 
\[
\int_K \phi(x) \, d\mu \geq \int_I \phi(x) \, d\nu
\]
holds whenever $\phi$ is convex. Note that here the set $K$ is centered and the convex function $\phi$ is not, whereas in our \autoref{th:momentnone} the moment function $|x|^q$ is centered but $K$ is not. 

Those authors also establish upper bounds, which naturally are of a different nature to the lower bounds in our paper. In particular they show in \cite[Theorem 2]{BLP11} that if $K \subset \R^2 \simeq \C$ is connected and has electrostatic centroid at the origin then 
\[
\int_K \phi(\Real z) \, d\mu(z) \leq \int_I \phi(\Real z) \, d\nu(z) .
\]
This result says intuitively that among connected planar sets of given capacity, the horizontal line segment is the ``most spread out'' as quantified by convex means of the horizontal variable.  

A further upper bound due to Laugesen \cite[Corollary 6]{L93} implies that if $K \subset \C$ is connected and contains the origin and has the same logarithmic capacity as an interval $I_0$ having one endpoint at the origin, then 
\[
\int_K \phi(\log |z|) \, d\mu \leq \int_{I_0} \phi(\log |z|) \, d\nu
\]
whenever $\phi$ is convex. Baernstein, Laugesen and Pritsker \cite[Corollary 6.3]{BLP11} deduce from this inequality that if $K \subset \C$ is connected and symmetric with respect to the origin and contains the origin, then 
\begin{equation} \label{eq:BLPconj}
\int_K \phi(\log |z|) \, d\mu \leq \int_I \phi(\log |z|) \, d\nu ,
\end{equation}
where this last integral is taken not over $I_0$ but over the centered interval $I$. 

They raise an open problem for logarithmic equilibrium moments \cite[Conjecture 2]{BLP11}: if $K \subset \C$ is connected, contains the origin, and has its electrostatic centroid at the origin, then \eqref{eq:BLPconj} holds whenever $\phi$ and $\phi^\prime$ are convex. A special case would be a long-standing conjecture of Pommerenke for conformal maps on the exterior of the disk. 

Lastly, upper bounds for $\log^+$ moments on Julia sets are presented by Pritsker \cite{P23}, with the maximizers being intervals in the plane. 

\subsection{Literature on isoperimetric theorems for Riesz capacity} Our work on minimizing moments of equilibrium measures is indirectly motivated by a classic extremal problem in mathematical physics: to minimize capacity among conductors of given volume. In the Newtonian and logarithmic situations ($p=n-2$), the Poincar\'{e}--Carleman--Szeg\H{o} theorem \cite{PS51} asserts that the set of given volume that minimizes capacity is the ball. In the plane, Solynin and Zalgaller \cite{SZ04} proved a beautiful and difficult polygonal analogue, that among $N$-sided polygons with given area, the one that minimizes logarithmic capacity is the regular $N$-gon. In a related vein, although much easier to prove, Laugesen \cite{L22} minimized capacity among linear images of sets (not necessarily polygons) that have rotational symmetry.  

Minimality of Riesz $p$-capacity for the ball among compact sets of given volume is known when $n-2<p<n$ by Watanabe \cite[p.{\,}489]{W83}; see also Betsakos \cite{B04a,B04b} and M\'{e}ndez--Hern\'{a}ndez \cite{MH06}. Rearrangement and polarization methods underlie these results, along with probabilistic characterizations of the capacity (related to $\alpha$-stable processes). The case $0<p<n-2$ remains open, as far as we know. For background, see the papers mentioned above and the comment by Mattila \cite[p.{\,}193]{M90}. A probabilistic approach could perhaps be tried to tackle \autoref{co:moments} for moments of Riesz equilibrium measures.

\section{\bf Proof of \autoref{th:momentnewton}}
\label{sec:momentnewtonproof}

The equilibrium measure $\nu$ for the ball $B$ is normalized surface area measure on the sphere $\partial B$, since  $p=n-2$ in this theorem; see \autoref{sec:background}.

\subsection*{Part (i).} Assume to begin with that $n \geq 3$. Recall that the equilibrium potential is bounded above everywhere by the energy (\autoref{upperu}), which at the origin means  
\begin{equation} \label{eq:potentialineq}
\int_K |x|^{-{(n-2)}} \, d\mu \leq V_{n-2}(K) ;
\end{equation}
equality holds for the centered ball $B$ since the origin is an interior point. 

Suppose $q>0$. Let $F(s)=s^{-q/(n-2)}$, so that $F$ is strictly convex and decreasing and has $|x|^q = F(1/|x|^{n-2})$. Then
\begin{align*}
\int_K |x|^q \, d\mu 
& = \int_K F(|x|^{-(n-2)}) \, d\mu \\
& \geq F \! \left( \int_K |x|^{-(n-2)} \, d\mu \right) && \text{by Jensen's inequality since $F$ is convex} \\
& \geq F \! \left( V_{n-2}(K) \right) && \text{by \eqref{eq:potentialineq} since $F$ is decreasing} \\
& = F \! \left( V_{n-2}(B) \right) && \text{because $K$ and $B$ have the same energy} \\
& = F \! \left( \int_B |x|^{-(n-2)} \, d\nu \right) && \text{by equality in \eqref{eq:potentialineq} for the ball} \\
& = \int_B |x|^q \, d\nu ,
\end{align*}
using in the final step that $|x|$ is constant on the sphere where $\nu$ is supported. 

Now suppose $q < -(n-2)$ and that the origin is a regular point of $K$, which means simply that \eqref{eq:potentialineq} holds with equality. Again $F$ is strictly convex, and even though it is increasing, the argument above remains valid since \eqref{eq:potentialineq} now holds with equality. 

Assume next that $n=2$. The potential satisfies the logarithmic analogue of \eqref{eq:potentialineq}, which is $\int_K \log 1/|x| \, d\mu \leq V_{log}(K)$, with equality when $K$ is the disk $B$ centered at the origin. We may argue as above using the strictly convex function $F(s)=e^{-q s}$, which is decreasing when $q>0$. The key relation is that $|x|^q = F(\log 1/|x|)$.

\subsection*{Part (ii).} For logarithmic moments, argue as for $q>0$ in part (i), except using the strictly convex, decreasing function $F(s) = \log 1/s$ when $n \geq 3$ and the linear decreasing function $F(s)=-s$ when $n=2$. The proof goes through as before, using that $F(1/|x|^{n-2}) = (n-2) \log |x|$ when $n \geq 3$ and $F(\log 1/|x|) = \log |x|$ when $n=2$. 

Alternatively, one could obtain part (ii) from part (i) by letting $q \searrow 0$ and obtaining the logarithm as an endpoint derivative, since $(|x|^q-1)/q \to \log |x|$ as $q \to 0$.

\subsection*{Part (iii).} Suppose $-(n-2)<q<0$. Argue as for $q>0$ in part (i), but now using the strictly convex, decreasing function $F(s)=-s^{-q/(n-2)}$, for which $-|x|^q = F(1/|x|^{n-2})$. One obtains that $-\int_K |x|^q \, d\mu \geq -\int_B |x|^q \, d\nu$. The same inequality holds when $q=-(n-2)<0$, by applying the potential inequality \eqref{eq:potentialineq} to $K$ and recalling that \eqref{eq:potentialineq} holds with equality for $B$. 

\subsection*{Part (iv).} When $q=-(n-2)$ with $n \geq 3$ we conclude directly from \eqref{eq:potentialineq} that equality holds in case (iii) if and only if the origin is a regular point of $K$, and similarly when $n=2$ we conclude from the logarithmic analogue of \eqref{eq:potentialineq} that equality holds in case (ii) if and only if the origin is a regular point of $K$. 

Next we prove the other equality statements.

``$\Longrightarrow$'' Suppose equality holds in part (i). To begin with, suppose $n \geq 3$. Examining the argument above, we notice equality must hold in Jensen's inequality. Strict convexity of $F$ then requires that $|x|$ is constant $\mu$-a.e.\ on $K$. Writing $R$ for that constant value, it follows that $\mu$ is supported in the sphere of radius $R$ centered at the origin. Equality of the moments in part (i) implies that $R$ is the radius of $B$, so that $\mu$ is supported in $\partial B$.

Further, since $F$ is strictly decreasing when $q>0$, equality in the proof of part (i) implies that $\int_K |x|^{-(n-2)} \, d\mu=V_{n-2}(K)$, and when $q<-(n-2)$ the same equation holds by hypothesis. Thus, either way, at the origin the equilibrium potential of $K$ achieves its maximum value, namely the energy of $K$. We claim $K$ contains the whole sphere $\partial B$. For if it omitted any point of the sphere then it would omit a whole neighborhood of that point and hence the support of $\mu$ would be confined to a proper compact subset $S$ of $\partial B$. The equilibrium potential of $K$ would then be harmonic and nonconstant on the connected open set $\Rn \setminus S$, and so by the strong maximum principle it would be impossible for the equilibrium potential to attain its maximum at the origin, because the origin is an interior point of that open set. This contradiction shows that $K$ must contain the whole sphere $\partial B$. 

Since $\mu$ is supported in $\partial B$, we see that $\mu$ must be the equilibrium measure for that sphere, which is  the normalized surface area measure. Thus $\mu=\nu$. Outside the ball $B$, the potential of $\nu$ is strictly smaller than the energy of $B$ by \autoref{upperu}. Equivalently, outside $B$ the potential of $\mu$ is strictly less than the energy of $K$. By the equality case of \autoref{upperu}, we conclude that $K \setminus B$ contains no interior points of $K$ and that it has inner $(n-2)$-capacity zero. Hence we have proved the desired equality statement. 

When $n=2$, the argument above for equality in case (i) is identical except using logarithmic energies, capacities and potentials. 

If equality holds in part (ii) with $n \geq 3$, or in part (iii) with $-(n-2) < q < 0$, then the arguments above again imply by the strict convexity of $F$ that $K$ contains the sphere $\partial B$ and $K \setminus B$ has inner $(n-2)$-capacity zero. 

``$\Longleftarrow$'' For the other direction of these equality statements, suppose $K \setminus B$ has inner $(n-2)$-capacity zero. If $K^\prime$ is any compact subset of $K \setminus B$ then $\capntwo(K^
\prime)=0$, and so $K^\prime$ has infinite energy, which implies $\mu(K^\prime)=0$ (because otherwise the normalized restriction of $\mu$ to $K^\prime$ would give finite energy). Hence $\mu(K^
\prime)=0$ for every compact subset of $\Rn \setminus B$, and so $\mu(\Rn \setminus B)=0$. Thus $\mu$ is supported in $B$, and since $K$ and $B$ are assumed to have the same energy, the uniqueness of equilibrium measure forces $\mu=\nu$. The moments of $\mu$ and $\nu$ therefore agree.

\section{\bf Proof of \autoref{pr:momentsthreshold}}
\label{sec:threshold}

The function $F(s)=s^{q/{q_*}}$ is convex since $q\geq q_*$ and so Jensen's inequality yields  
\[
\int_K |x|^q \, d\mu \geq \left(\int_K |x|^{q_*}\, d\mu \right)^{\! \! q/q_*} .
\] 
Since the moment inequality \eqref{eq:momentsthreshold} holds for $q_*$, we see that 
\[
\left(\int_K |x|^{q_*} \, d\mu \right)^{\! \! q/q_*} \geq \left(\int_B |x|^{q_*} \, d\nu \right)^{\! q/q_*} = \int_B |x|^q \, d\nu ,
\]
where we used in the final equality that for $0<p<n-2$, the Riesz equilibrium measure $\nu$ on the ball $B$ is normalized surface measure on the boundary sphere \cite[Theorem 4.6.7]{BHS19}, so that $|x|$ is constant on the support of $\nu$.

\section{\bf Moment formula in $\Rnp$} 
The moments of equilibrium measure when $p=n-1$ will be expressed in terms of the equilibrium potential in $\Rnp$, in \autoref{le:Phimoment} below. First we need the mean value of the Newtonian potential near infinity, in $n+1$ dimensions.
\begin{lemma}[Spherical mean value of the potential near infinity] \label{le:meanvalue}
If the compact set $K \subset \Rnp$ has positive $(n-1)$-capacity then its equilibrium potential $u$ satisfies  
\[
\frac{1}{|\Sphn|} \int_{\Sphn} u(r\xi) \, d\xi = 
\begin{cases}
1/r^{n-1} , & n \geq 2 , \\
\log 1/r , & n=1 ,
\end{cases}
\]
whenever $r$ is large enough that $K \subset \B^{n+1}(r)$. 
\end{lemma}
Here $d\xi$ is the surface area element on the sphere $\Sphn$. In what follows, we write $x \in \Rn$ and $\xh=(x,z)$ for a typical point in $\Rnp$, as depicted in \autoref{fig:K} earlier in the paper. 
\begin{proof}
Suppose $n \geq 2$ and $K \subset \Rnp$. The potential 
\begin{equation} \label{eq:pot1}
u(r\xi) = \int_K \frac{1}{|r\xi-\xh|^{n-1}} \, d\mu(\xh) 
\end{equation}
decays like $1/r^{n-1}+O(1/r^n)$ as $r \to \infty$. After integrating with respect to $d\xi/|\Sphn|$, the same decay rate holds for the spherical mean value of $u$. On the other hand, the spherical mean value of $u$ is a radial harmonic function of $r$, when $r$ is large enough that $\B^{n+1}(r) \supset K$, and so it must have the form $c_1/r^{n-1}+c_0$. Comparing with the known decay rate, we deduce $c_1=1$ and $c_0=0$. 

Now suppose $n=1$ and $K \subset \R^2$. The logarithmic potential 
\begin{equation} \label{eq:pot2}
u(r\xi) = \int_K \log \frac{1}{|r\xi-\xh|} \, d\mu(\xh)
\end{equation}
behaves like $\log (1/r)+o(1)$ as $r \to \infty$, and so after integrating with respect to $d\xi/|\Sph^1|$, the same behavior holds for the  mean value of $u$ over the circle of radius $r$. But when $r$ is large enough that $\B^2(r) \supset K$, the circular mean value is a radial harmonic function in $\R^2$ and so has the form $c_1 \log (1/r)+c_0$. The behavior at infinity implies that $c_1=1$ and $c_0=0$. 
\end{proof}

Define
\[
a_n = 
\begin{cases}
1/(n-1)|\Sphn| , & n \geq 2 , \\
1/2\pi , & n=1 .
\end{cases}
\]
\begin{lemma}[Difference of moments] \label{le:Phimoment}
Let $n \geq 1$. Assume $\Phi(r)$ is $C^3$-smooth for $r>0$ with $\Phi^{\prime\prime}(r)=O(1/r^{n+1-\delta})$ as $r\to 0$, for some $\delta>0$. Suppose $K$ and $B$ are compact sets in $\Rnp$ each having positive $(n-1)$-capacity. 

If $R>0$ is large enough that both sets lie in the ball $\B^{n+1}(R)$ then the difference of their $\Phi$-moments can be expressed as 
\begin{align}
\int_K \Phi(|\xh|) \, d\mu(\xh) - \int_B \Phi(|\xh|) \, d\nu(\xh) 
& = a_n \int_{\B^{n+1}(R)} \Psi(|\xh|) (v(\xh) - u(\xh)) \, d\xh \label{eq:momentvu} \\
& = a_n \bigg{(} \Psi(R)\int_{\B^{n+1}(R)} (v(\xh)-u(\xh))\, d\xh \label{eq:momentJ} \\
&  \hspace*{0.5cm} -\int_0^R \Psi^\prime (r) \int_{\B^{n+1}(r)}  (v(\xh)-u(\xh))\, d\xh \, dr \bigg{)} \notag
\end{align}
where $\mu$ and $\nu$ are the $(n-1)$-equilibrium measures for $K$ and $B$ in $\Rnp$, with corresponding equilibrium potentials $u$ and $v$, and $\Psi(r)=\Phi^{\prime\prime}(r)+(n/r)\Phi^\prime(r)$ is the Laplacian of $\Phi(r)$ in $\Rnp$. 
\end{lemma}
For the second moment ($q=2$), the lemma simplifies to say 
\begin{align*}
\int_K |\xh|^2 \, d\mu(\xh) - \int_B |\xh|^2 \, d\nu(\xh) 
& = 2(n+1)a_n \int_{\B^{n+1}(R)} (v(\xh) - u(\xh)) \, d\xh ,
\end{align*}
by taking $\Phi(r)=r^2$ in \eqref{eq:momentvu} and noting $\Psi(r) \equiv 2(n+1)$. 
\begin{proof}[Proof of formula \eqref{eq:momentvu} in the lemma] The Laplacian of the potential $u$ in \eqref{eq:pot1} and \eqref{eq:pot2} is proportional to the equilibrium measure, with 
\[
- a_n \Delta u = d\mu
\]
in the distributional sense since the fundamental solution for the Laplacian in $\Rnp$ is $a_n/|\xh|^{n-1}$ when $n \geq 2$ and $a_1 \log 1/|\xh|$ when $n=1$. 

Suppose initially that $\Phi(|\xh|)$ is $C^2$-smooth on $\Rnp$, including at the origin. By the distributional Laplacian formula, 
\begin{align}
& \frac{1}{a_n} \left( \int_K \Phi(|\xh|) \, d\mu(\xh) - \Phi(R) \right) \label{eq:distr} \\
& = \int_{\B^{n+1}(R)} \big( \Phi(R) - \Phi(|\xh|) \big) \, \frac{d\mu(\xh)}{-a_n} \notag \\
& = \int_{\B^{n+1}(R)} \Delta \big(\Phi(R)-\Phi(|\xh|) \big) u(\xh) \, d\xh - \int_{\Sphn(R)} \frac{\partial\ }{\partial \nu} \big( \Phi(R)-\Phi(|\xh|) \big) u(\xh) \, dS(\xh) \notag \\
& = - \int_{\B^{n+1}(R)} \Psi(|\xh|) u(\xh) \, d\xh + \Phi^\prime(R) \int_{\Sphn(R)} u \, dS \notag \\
& = - \int_{\B^{n+1}(R)} \Psi(|\xh|) u(\xh) \, d\xh + \Phi^\prime(R) |\Sphn(R)| \label{eq:repformula}
\begin{cases}
1/R^{n-1} , & n \geq 2 , \\
\log 1/R , & n=1 ,
\end{cases}
\end{align}
after using \autoref{le:meanvalue} in the final equality. 

We claim that formula \eqref{eq:repformula} continues to hold if the $C^2$-function $\Phi$ is not smooth at the origin and instead satisfies there the hypothesis in the lemma, namely that $\Phi^{\prime\prime}(r)=O(1/r^{n+1-\delta})$ as $r\to 0$. By choosing $\delta$ smaller if necessary, we may suppose $\delta<1$.

To begin with, integrating the hypothesis on $\Phi^{\prime\prime}$ implies that $\Phi^\prime(r)=O(1/r^{n-\delta})$. Integrating again shows $\Phi(r)=O(1/r^{n-1-\delta})$ when $n \geq 2$ and $\Phi(r)=O(1)$ when $n=1$. Thus when $n \geq 2$, the $\Phi$-moment in the lemma is well-defined, because near the origin $\Phi(|\xh|)$ grows slower than $1/|\xh|^{n-1}$, and the integral of $1/|\xh|^{n-1}$ with respect to $\mu$ is simply the Newtonian potential at the origin, which is finite. When $n=1$, $\Phi$ is bounded near the origin and so the $\Phi$-moment is obviously finite. Hence for each $n$, the integral in \eqref{eq:distr} is well defined.

The integral in \eqref{eq:repformula} is well defined too, under our hypothesis on $\Phi^{\prime\prime}$, since the potential $u$ is locally bounded while the estimates on $\Phi^{\prime\prime}$ and $\Phi^\prime$ near the origin ensure that $\Psi(r)=\Phi^{\prime\prime}(r)+n\Phi^\prime(r)/r=O(r^{-n-1+\delta})$ near $r=0$, so that $\Psi(|\xh|)$ is locally integrable on $\Rnp$.

To confirm that \eqref{eq:distr} and \eqref{eq:repformula} are still equal, under the hypothesis on $\Phi^{\prime\prime}$, we let $\e>0$ and consider the regularized function 
\[
\Phi_\e(r) = \Phi(\sqrt{\e^2+r^2}) ,
\]
noting that $\Phi_\e(|\xh|)$ is $C^2$-smooth on $\Rnp$. Applying \eqref{eq:distr} and \eqref{eq:repformula} to $\Phi_\e$ yields formulas we call \eqref{eq:distr}${}_\e$ and \eqref{eq:repformula}${}_\e$. By dominated convergence, \eqref{eq:distr}${}_\e$ approaches \eqref{eq:distr} as $\e \to 0$, with a dominator (which is needed when $n \geq 2$) being provided by our estimate on $\Phi$:
\[
|\Phi_\e(|\xh|)| = \Phi(\sqrt{\e^2+|\xh|^2}) \leq \frac{C}{(\e^2+|\xh|^2)^{(n-1-\delta)/2}} \leq \frac{C}{|\xh|^{n-1-\delta}} ,
\]
which is integrable with respect to $\mu$ by finiteness of the equilibrium potential at the origin. 

In \eqref{eq:repformula}${}_\e$, the factor $\Phi_\e^\prime(R)$ clearly converges to $\Phi^\prime(R)$. Thus the  task is to establish the following limit for the integral in \eqref{eq:repformula}${}_\e$: 
\[
\lim_{\e \to 0} \int_{\B^{n+1}(R)} \Psi_\e(|\xh|) u(\xh) \, d\xh = \int_{\B^{n+1}(R)} \Psi(|\xh|) u(\xh) \, d\xh .
\]
Since the potential $u$ is locally bounded, in order to invoke dominated convergence we need only show $\Psi_\e \to \Psi$ pointwise as $\e \to 0$ and find an appropriate dominator for $\Psi_\e$. By direct computation,  
\begin{align*}
\Psi_\e(r) 
& = \Phi_\e^{\prime\prime}(r)+\frac{n}{r} \Phi_\e^\prime(r) \\
& = \Phi^{\prime\prime}(\sqrt{\e^2+r^2}) \frac{r^2}{\e^2+r^2} + \frac{\Phi^\prime(\sqrt{\e^2+r^2})}{\sqrt{\e^2+r^2}} \left( \frac{\e^2}{\e^2+r^2} + n \right) \\
& \to \Phi^{\prime\prime}(r)+\frac{n}{r} \Phi^\prime(r) = \Psi(r)
\end{align*}
as $\e \to 0$, thus giving the desired pointwise limit. Observe  
\[|
\Psi_\e(r)| \leq |\Phi^{\prime\prime}(\sqrt{\e^2+r^2})| + \frac{|\Phi^\prime(\sqrt{\e^2+r^2})|}{\sqrt{\e^2+r^2}} (1+n) 
\]
and hence 
\[
|\Psi_\e(|\xh|)| \leq \frac{C}{(\e^2+|\xh|^2)^{(n+1-\delta)/2}} \leq \frac{C}{|\xh|^{n+1-\delta}}
\]
by our estimates on $\Phi^{\prime\prime}$ and $\Phi^\prime$. This last quantity $C/|\xh|^{n+1-\delta}$ is Lebesgue integrable on the ball $\B^{n+1}(R)$, as needed for a dominator. 

Now that \eqref{eq:distr} and \eqref{eq:repformula} are known to agree, we may subtract the analogous expressions for the set $B$ with measure $\nu$ and potential $v$, hence arriving at the formula \eqref{eq:momentvu} in the lemma. 
\end{proof}
\begin{proof}[Proof of formula \eqref{eq:momentJ}] For the second formula in the lemma, assume in addition that $\Phi$ is $C^3$-smooth, so that we may differentiate $\Psi$ in the following argument. Start by rewriting the right side of \eqref{eq:momentvu} using spherical coordinates as
\begin{align*}
& \int_{\B^{n+1}(R)} \Psi(|\xh|) (v(\xh) - u(\xh)) \, d\xh \\
& = \int_0^R \Psi(r) \left( \frac{d}{dr} \int_{\B^{n+1}(r)} (v(\xh) - u(\xh)) \, d\xh \right) \! dr. \\
\end{align*}
Integrating by parts yields formula \eqref{eq:momentJ}, since the boundary term at $r=0$ vanishes as follows:
\[
 \Psi(r)\int_{\B^{n+1}(r)}(v(\xh)-u(\xh))\,d\xh=O( r^{-n-1+\delta})\cdot O(r^{n+1}) \to 0
\]
as $r \to 0$, where we used that $v-u$ is bounded on $\B^{n+1}(R)$.
\end{proof}

\section{\bf Integral operator $J$}
\label{sec:J}

An important tool for the next few sections is the integral operator   
\[
(Ju)(r,z) = \int_{\B^n(r)} u(x,z) \, dx , \qquad r \geq 0, \quad z \in \R ,
\]
acting on a locally integrable function $u$ on $\Rnp$. Note that $Ju$ integrates over an $n$-dimensional ball centered in the slice at height $z$, as illustrated earlier in \autoref{fig:K}. 

Operators of this nature and more sophisticated variants in which one first replaces $u$ by its symmetric decreasing rearrangement on the slice have long been employed for elliptic and parabolic comparison theorems in a variety of coordinate systems, including rectangular, polar and spherical coordinates and the cylindrical coordinates needed in this paper. See for example the work of Alvino, Lions and Trombetti \cite[{\S}III]{ALT91}, the book by Baernstein \cite[Chapter 10]{B19}, and the survey by Talenti \cite[{\S\S}6--8]{T16}. Talenti's results in the 1970s and '80s were particularly influential. 

We borrow the $J$ notation from Baernstein, whose $\s$-function technique yields many sharp inequalities in complex analysis; see \cite[Chapter 11]{B19} and references therein. His approach emphasizes the role of commutation relations such as in the following lemma.   

Define a differential operator $\deltastar$  acting on functions $w(r,z)$ by 
\begin{align}
\deltastar w 
& = \partial_{rr} w - \frac{n-1}{r} \partial_r w + \partial_{zz} w \notag \\
& = r^{n-1} \partial_r \left( r^{1-n} \partial_r w \right) + \partial_{zz} w . \label{eq:Jmoment2}
\end{align}
\begin{lemma}[Commutation relation] \label{le:commute}
The $J$ operator ``commutes'' with the Laplacian $\Delta$ on $\Rnp$, in the sense that 
\[
\deltastar J = J \Delta .
\]
\end{lemma}
\begin{proof}
Suppose $u \in C^2(\Rnp)$. Rewriting $Ju$ using spherical coordinates as 
\[
Ju(r,z) = \int_0^r \int_{\Sph^{n-1}} u(s\zeta,z) \, d\zeta \, s^{n-1} \, ds ,
\]
we find that formula \eqref{eq:Jmoment2} for the $\deltastar$ operator implies  
\begin{align*}
\deltastar (Ju(r,z))
& = \int_{\Sph^{n-1}} \partial_r u(r\zeta,z) r^{n-1} \, d\zeta + \int_0^r \int_{\Sph^{n-1}} \partial_{zz} u(s\zeta,z)\, d\zeta \, s^{n-1} \, ds \\
& = \int_{\partial \B^n(r)} \frac{\partial u}{\partial \nu}(x,z) \, dS(x) + \int_{\B^n(r)} \partial_{zz} u(x,z) \, dx .
\end{align*}
Applying the $n$-dimensional divergence theorem to the first term yields that 
\begin{align*}
\deltastar (Ju(r,z))
& =  \int_{\B^n(r)} \Delta_x u(x,z) \, dx + \int_{\B^n(r)} \partial_{zz} u(x,z) \, dx \\
& = \int_{\B^n(r)} \Delta u(x,z) \, dx = J (\Delta u)(r,z) ,
\end{align*}
so that $\deltastar J = J \Delta$. 
\end{proof}
Later in the paper we will deal with functions $u$ that are $C^2$-smooth only when $z \neq 0$. \autoref{le:commute} and its proof are local in $z$ and so continue to hold in that case.

\section{\bf Nonnegativity of $Jv-Ju$}

The difference of moments for $K$ and $B$ is expressed by formula \eqref{eq:momentJ} in terms of the difference $v-u$ of potentials. To prove this difference of moments is nonnegative, for \autoref{th:momentnone}, we start by showing $Jv \geq Ju$.

Consider a compact set $K\subset \Rn$. When $K$ is regarded as lying in $\Rnp$, the $(n-1)$-equilibrium potential $u$ is harmonic on $\Rnp \setminus K$. In particular, $u(x,z)$ is harmonic on the upper halfspace, where $z>0$.
\begin{proposition} \label{pr:JvJuPos} Let $n \geq 1$. If a compact set $K \subset \Rn$ has the same $(n-1)$-capacity as a closed ball $B \subset \Rn$ centered at the origin, $\capnone(K)=\capnone(B)>0$, then
\[
 J(v-u)(r,z)\geq 0,
 \]
 for $r\geq 0, z\in \mathbb{R}$, where $u$ and $v$ are the equilibrium potentials of $K$ and $B$ in $\Rnp$.
\end{proposition}

We prove \autoref{pr:JvJuPos} in several steps.

\subsection*{Step 1: $Jv-Ju$ extends continuously to $[0,\infty] \times [0,\infty]$}
Let 
\[
w=Jv-Ju \qquad \text{for\ } (r,z) \in [0,\infty) \times [0,\infty) .
\]
Clearly $w$ is continuous when $z>0$, because $v(x,z)$ and $u(x,z)$ are smooth away from the sets $B$ and $K$, respectively, and those sets lie in the slice at height $z=0$. For continuity at $z=0$, note the potentials $u$ and $v$ are locally bounded and are continuous at $(x,0)$ for almost every $x \in \Rn$ by the potential theoretic properties in \autoref{sec:background} (take $m=n+1$ there), so that dominated convergence yields continuity of $Ju$ and $Jv$ at $(r,0)$ for every $r \geq 0$. Thus $w$ is continuous everywhere in the closed first quadrant of the $rz$-plane. We wish to extend it continuously to $r=\infty$ and $z=\infty$. 

\begin{lemma} \label{le:wcont1}
$w(r,z)$ extends continuously to $[0,\infty) \times [0,\infty]$, equalling $0$ when $z=\infty$.
\end{lemma}
The decay of $v-u$ near infinity is not enough by itself to extend $Jv-Ju$ successfully. Cancellations due to the integration of $v$ and $u$ over the ball $\B^n(r)$ are needed too, as we shall see in the proof. 
\begin{proof}
Extend $w$ to the ``top'' side of its domain, where $z=\infty$, by defining $w(r,\infty)=0$ for $r \in [0,\infty)$. To prove that this extended $w$ is continuous, we show 
\begin{equation} \label{eq:Jv-Ju}
w(R,z) = Jv(R,z)-Ju(R,z)=O(1/z) \qquad \text{as $z \to \infty$,}
\end{equation}
uniformly with respect to $R \in [0,\infty)$. 

First we prove (\ref{eq:Jv-Ju}) when $n=1$, for $K \subset \R$. Its logarithmic equilibrium potential is 
\begin{align}
u(x,z) & = \int_K \log\frac{1}{|(x,z)-(y,0)|} \, d\mu(y)  \notag \\
& = \frac{1}{2} \int_K \log \frac{1}{x^2+z^2+y^2-2xy} \, d\mu(y)  \label{eq:n=1potential} \\
& = \frac{1}{2}\log\frac{1}{x^2+z^2}-\frac{1}{2}\int_K \log\left(1+\frac{y^2-2xy}{x^2+z^2} \right) d\mu(y) \notag \\
& = \frac{1}{2}\log\frac{1}{x^2+z^2}-\frac{1}{2}\int_K \frac{y^2-2xy}{x^2+z^2} \, d\mu(y) + O \left( \frac{1+|x|}{x^2+z^2} \right)^{\!\!2} \notag 
\end{align}
since $\log(1+t)=t+O(t^2)$ when $|t|$ is small; note in the remainder term that we used the boundedness of $|y|$ when $y \in K$. Hence 
\begin{equation} \label{eq:n=1expansion}
u(x,z) = \frac{1}{2}\log\frac{1}{x^2+z^2}+\frac{x}{x^2+z^2}\int_Ky\, d\mu(y)+O\left(\frac{1}{x^2+z^2}\right) 
\end{equation}
for all $x \in \R$, provided $z \geq 1$ is large enough that the remainder term $O(1/(x^2+z^2))$ is suitably small. The function $x \mapsto x/(x^2+z^2)$ is odd and so the second term in (\ref{eq:n=1expansion}) integrates to $0$ over $x \in (-R,R)$. Additionally, the first term in (\ref{eq:n=1expansion}) is the same for both $K$ and $B$, and so it cancels when we take the difference $v-u$. Hence 
\begin{align*}
|Jv(R,z)-Ju(R,z)|
& = \left| \int_{-R}^R (v(x,z)-u(x,z))\, dx \right| \\
& \leq O\left( \int_{-\infty}^{\infty} \frac{1}{x^2+z^2} \, dx \right)
 =O\!\left(\frac{1}{z}\right) ,
\end{align*}
as $z \to \infty$, by evaluating the integral. This completes the proof of \eqref{eq:Jv-Ju} for $n=1$.

Next we carry out the analogous computations for $n\geq 2$. The equilibrium potential of $K$ at the point $(r\zeta,z) \in \Rnp$ (where the spherical coordinates on $\Rn$ are $x=r\zeta, r \geq 0, \zeta \in \Sph^{n-1}$) is  
\begin{align}
& u(r\zeta,z) \notag \\
& = \int_K \frac{1}{|(r\zeta,z)-(y,0)|^{n-1}} \, d\mu(y) \notag \\
& = \int_K \frac{1}{(r^2+z^2+|y|^2-2r\zeta \cdot y )^{(n-1)/2}} \, d\mu(y) \label{eq:potential} \\ 
& = \frac{1}{(r^2+z^2)^{(n-1)/2}} \int_K \left( 1+\frac{|y|^2-2r\zeta \cdot y }{r^2+z^2} \right)^{\! \! -(n-1)/2} \! d\mu(y) \notag \\ 
& = \frac{1}{(r^2+z^2)^{(n-1)/2}} \left\{ \int_K \left( 1 - \frac{n-1}{2} \, \frac{|y|^2-2r\zeta \cdot y}{r^2+z^2} \right) d\mu(y) + O \! \left( \frac{1+r}{r^2+z^2}\right)^{\!\!2}  \right\} \notag
\end{align}
by the binomial approximation, again using that $|y|$ is bounded.  Integrating $\zeta \in \Sph^{n-1}$ with respect to the surface measure element $d\zeta$ on the sphere eliminates the term $\zeta \cdot y$ from the integrand, because it has integral zero. Thus 
\begin{equation} \label{eq:mean1}
\int_{\Sph^{n-1}} u(r\zeta,z) \, d\zeta = \frac{|\Sph^{n-1}|}{(r^2+z^2)^{(n-1)/2}} \left\{ 1 + O \! \left( \frac{1}{r^2+z^2} \right) \right\} 
\end{equation}
where we also simplified the remainder term by estimating $(1+r)^2$ with $O(r^2+z^2)$, assuming $z \geq 1$. The same estimate holds when $K$ is replaced by the set $B$, and so subtracting the two formulas gives that 
\begin{equation} \label{eq:mean2}
\int_{\Sph^{n-1}} \left( v(r\zeta,z) - u(r\zeta,z) \right) d\zeta = O \! \left( \frac{1}{(r^2+z^2)^{\! (n+1)/2}} \right) 
\end{equation}
for all $r \in [0,\infty)$, provided $z \geq 1$ is large enough that the remainder term on the right side of \eqref{eq:mean2} is small enough to justify the earlier application of the binomial approximation. Multiplying by $r^{n-1}$ and integrating from $0$ to $R$ yields that 
\[
\left| Jv(R,z) - Ju(R,z) \right| \leq O \left( \int_0^\infty \frac{r^{n-1}}{(r^2+z^2)^{\! (n+1)/2}} \, dr \right) = O \! \left( \frac{1}{z} \right)
\]
as $z \to \infty$,  by changing variable with $r \mapsto rz$ in the integral. This uniform estimate completes the proof of \eqref{eq:Jv-Ju} when $n \geq 2$, and hence finishes the lemma. 
\end{proof}

\begin{lemma} \label{le:wcont2}
$w(r,z)$ extends continuously to $[0,\infty]\times [0,\infty]$, that is, to $r=\infty$. 
\end{lemma}
\begin{proof} We will prove below that positive constants $C,R$ exist such that   
\begin{equation} \label{eq:wcont}
w(r,z)  = Jv(R,z) - Ju(R,z) + \int_R^r Q(s,z) \, ds 
\end{equation}
for $r>R$ and $z\geq 0$, where the quantity $Q(s,z)$ is continuous and satisfies $|Q(s,z)| \leq C/s^2$. Dominated convergence then confirms that $w(r,z_*)$ possesses a limiting value as $r \to \infty$ and $z_* \to z \in [0,\infty)$, and that this extension of $w$ to $r=\infty$ is continuous at each point $(\infty,z)$ with $z \in [0,\infty)$. The value of $w$ at that point can be found by taking the limit as $r \to \infty$ with $z$ fixed, so that 
\[
w(\infty,z) = \lim_{r \to \infty} w(r,z) = \int_0^\infty \int_{\Sph^{n-1}} \left( v(s\zeta,z) - u(s\zeta,z) \right) d\zeta \, s^{n-1} \, ds . 
\]
When $n=1$ in this formula, we mean $\zeta \in \Sph^0= \{-1,+1\}$ and the $\zeta$-integral sums over these two values. 

We must still extend $w$ to the ``top right corner'', where $r=\infty$ and $z=\infty$. We showed in \eqref{eq:Jv-Ju} that $w(r,z)=O(1/z)$ as $z \to \infty$, with the $O(1/z)$ term being bounded independently of $r \in [0,\infty)$. By letting $r \to \infty$ we get also that $w(\infty,z)=O(1/z)$. Thus defining $w(\infty,\infty)=0$  completes the continuous extension of $w$ to $[0,\infty] \times [0,\infty]$. 

Let us turn now to proving \eqref{eq:wcont}. First consider the case $n=1$. If $|x| \geq 1$ then $(1+|x|)/(x^2+z^2) \leq 2/|x|$, and so by repeating the derivation of estimate \eqref{eq:n=1expansion} we find for $|x|$ sufficiently large that  
\begin{equation} \label{eq:dagger1}
u(x,z)=\frac{1}{2}\log\frac{1}{x^2+z^2}+\frac{x}{x^2+z^2}\int_Ky\, d\mu(y)+O\left(\frac{1}{x^2}\right)
\end{equation}
where the remainder bound is uniform in $z \geq 0$. That is, positive constants $C$ and $R$ exist such that \eqref{eq:dagger1} holds with $|O(1/x^2)| \leq C/x^2$ whenever $|x|>R, z \geq 0$. Hence
\begin{align*}
w(r,z) & = Jv(r,z)-Ju(r,z) \\
& = \int_{-R}^{R} (v(x,z)-u(x,z)) \, dx + \int_{R<|x|<r} (v(x,z)-u(x,z)) \, dx \\
& = Jv(R,z)-Ju(R,z) + \int_{R<|x|<r} O(1/x^2) \, dx,
\end{align*}
which proves \eqref{eq:wcont}.

Next we handle the case $n\geq 2$. If $r \geq 1$ then $(1+r)/(r^2+z^2) \leq 2/r$ and so repeating the argument for \eqref{eq:mean1} shows that for $r$ sufficiently large, 
\begin{equation} \label{eq:dagger2}
\int_{\Sph^{n-1}} u(r\zeta,z) \, d\zeta = \frac{|\Sph^{n-1}|}{(r^2+z^2)^{(n-1)/2}} \left\{ 1 + O \! \left( \frac{1}{r^2} \right) \right\} 
\end{equation}
where the remainder term is uniform in $z \geq 0$. That is, constants $C,R>0$ exist such that \eqref{eq:dagger2} holds with $|O(1/r^2)| \leq C/r^2$ whenever $r>R, z \geq 0$. By subtracting the analogous expression for the set $B$, we find 
\[
\int_{\Sph^{n-1}} \left( v(r\zeta,z) - u(r\zeta,z) \right) \, d\zeta =  \frac{1}{(r^2+z^2)^{(n-1)/2}} \, O \! \left( \frac{1}{r^2} \right) = O \! \left( \frac{1}{r^{n+1}} \right) 
\]
for $r>R$ and $z \geq 0$. The integral for $w$ can therefore be written in spherical coordinates as 
\begin{align*}
w(r,z) & = Jv(r,z)-Ju(r,z) \\
& = \int_{\B^n(R)} (v(x,z)-u(x,z)) \, dx + \int_R^r \int_{\Sph^{n-1}} \left( v(s\zeta,z) - u(s\zeta,z) \right) \, d\zeta \, s^{n-1} \, ds \\
& = Jv(R,z) - Ju(R,z) + \int_R^r O(1/s^2) \, ds 
\end{align*}
for $r>R$ and $z\geq 0$, which proves \eqref{eq:wcont}.
 \end{proof}

\subsection*{Step 2: the extended function $w=Jv-Ju$ equals $0$ on the left and right sides ($r=0,\infty$) and at the top ($z=\infty$)}
By definition, $Jv$ and $Ju$ vanish on the left side of their domain, where $r=0$ and $0 \leq z < \infty$, and in Step 1 we defined $w=0$ on the top side, where $r \in [0,\infty]$ and $z=\infty$. The next lemma shows $w=0$ on the right side, where $r=\infty$. 
\begin{lemma} \label{le:wlemma}
$w(\infty,z)=0$ when $0 \leq z < \infty$.
\end{lemma}
\subsubsection*{Remark}
The difference of potentials $v-u$ is not integrable on $\Rnp$: the individual potentials decay slowly at infinity ($\sim 1/|\xh|^{n-1}$), and one can calculate that while taking the difference improves the decay rate, it does not improve enough to gain integrability. Thus the integral $\int_\Rnp (v-u) \, d\xh$ does not exist, and so although \autoref{le:wlemma} can be integrated to conclude that the iterated integral $\int_{-\infty}^\infty \lim_{r \to \infty} \int_{\B^n(r)} (v(x,z)-u(x,z)) \, dxdz$ equals $0$, we cannot apply Fubini's theorem to say the same about $\int_\Rnp (v-u) \, d\xh$. 
\begin{proof}[Proof of \autoref{le:wlemma}]
We will prove $(d^2/dz^2) w(\infty,z) = 0$ when $0 < z < \infty$, so that $w(\infty,z)$ is linear. Since it tends to zero as $z \to \infty$, by Step 1, the linear function $w(\infty,z)$ must equal $0$ for all $z$, as desired.

To develop a formula for the $z$-derivative of $w$, first differentiate the potential $u$ given in \eqref{eq:potential} for $n \geq 2$ and in \eqref{eq:n=1potential} for $n=1$, finding for $x \in \Rn$ and $z>0$ that 
\begin{equation} \label{eq:zderiv}
\partial_z u(x,z) = -b_n z \int_K \frac{1}{|(x,z)-(y,0)|^{n+1}} \, d\mu(y) < 0
\end{equation}
where
\[
b_n = 
\begin{cases}
n-1, & n \geq 2 , \\
1  , & n=1 .
\end{cases}
\]
Hence constants $C$ and $R$ exist such that 
\begin{equation} \label{eq:Czr}
|\partial_z u(x,z)| \leq \frac{Cz}{|x|^{n+1}} , \qquad |x|>R, \quad z>0 . 
\end{equation}
Similarly, the second derivative satisfies 
\[
|\partial_{zz} u(x,z)| \leq \frac{C(|x|^2+z^2)}{|x|^{n+3}} , \qquad |x|>R, \quad z>0 . 
\]
These estimates hold also for the potential $v$, providing integrable dominators (with respect to $x \in \Rn$) for the first and second $z$-derivatives of $u$ and $v$. Hence we may differentiate successively through the iterated integral that defines $w(\infty,z)$ to obtain that 
\begin{align*}
\frac{d^2\ }{dz^2} w(\infty,z) 
& = \frac{d^2\ }{dz^2} \int_0^\infty \int_{\Sph^{n-1}} \left( v(r\zeta,z) - u(r\zeta,z) \right) d\zeta \, r^{n-1} \, dr \\
& = \int_0^\infty \frac{\partial^2 \ }{\partial z^2} \int_{\Sph^{n-1}} \left( v(r\zeta,z) - u(r\zeta,z) \right) d\zeta \, r^{n-1} \, dr \\
& = \int_0^\infty \int_{\Sph^{n-1}} \left( - \Delta_x v(x,z) + \Delta_x u(x,z) \right) d\zeta \, r^{n-1} \, dr 
\end{align*}
since $\Delta_x u + \partial_{zz} u=0$ and $\Delta_x v + \partial_{zz} v=0$ by harmonicity of the potentials in $\Rnp \setminus K$. Applying the divergence theorem yields that  
\begin{align*}
\frac{d^2\ }{dz^2} \, w(\infty,z) 
& = \lim_{R \to \infty}\int_{\B^n(R)} ( - \Delta_x v + \Delta_x u ) \, dx \\
& = \lim_{R \to \infty} \int_{\Sph^{n-1}} \left( - \frac{\partial v}{\partial r}(R\zeta,z) + \frac{\partial u}{\partial r}(R\zeta,z) \right) \! R^{n-1} \, d\zeta = 0 ,
\end{align*}
since expressions \eqref{eq:n=1potential} and \eqref{eq:potential} for the potential imply that $\partial u/\partial r$ and $\partial v/\partial r$ decay like $1/r^n$ as $r \to \infty$ with $z$ fixed. 
\end{proof}

\subsection*{Step 3: $w=Jv-Ju \geq 0$ in $[0,\infty] \times [0,\infty]$ } 
Since $w$ is continuous on the compact set $[0,\infty] \times [0,\infty]$ by Step 1, one of three cases must occur:
\begin{itemize}
\item[(a)] $w$ attains its minimum on the left or right sides ($r=0,\infty$ and $0 \leq z < \infty$) or the top ($0 \leq r \leq \infty$ and $z=\infty$), 
\item[(b)] $w$ attains its minimum at an interior point ($0<r<\infty$ and $0<z<\infty$),  
\item[(c)] $w$ attains its minimum on the bottom ($0 < r < \infty$ and $z = 0$).
\end{itemize}

In case (a), the minimum value is $0$ by Step 2, so that $Jv-Ju \geq 0$ as desired. 

In case (b), the strong minimum principle for the elliptic operator $\deltastar$ requires $w$ to be constant, because $w$ attains an interior minimum and 
\[
\deltastar w = \deltastar J(v-u) = J \Delta (v-u) = 0
\]
by \autoref{le:commute} and harmonicity of the potentials on $\{z > 0\}$. Since $w=0$ on the sides, the constant value of $w$ must be zero and so $w \equiv 0$. That is, $Jv-Ju \equiv 0$. 

From now on we consider case (c). Write $(r,0)$ for a point at which the minimum of $w$ is attained. We need to show $w(r,0) \geq 0$. 

Suppose $0 < r \leq \rho$, where $\rho$ is the radius of $B$ in $\Rn$. When $n \geq 2$, the equilibrium potential for the ball satisfies $v(x,0) = V_{n-1}(B)$ for $|x| \leq \rho$ and the potential for $K$ has $u(x,0) \leq V_{n-1}(K)$ for all $x$ by \autoref{upperu} (taking $m=n+1$ there). When $n=1$, the corresponding results are that $v(x,0) = V_{log}(B)$ for $|x| \leq \rho$ and $u(x,0) \leq V_{log}(K)$ for all $x$. Since the energies of $B$ and $K$ are assumed in \autoref{pr:JvJuPos} to be equal, we have for all $n$ that $v(x,0) \geq u(x,0)$ whenever $|x| \leq \rho$. Hence the minimum value of $w$ is 
\[
w(r,0)=\int_{\B^n(r)} (v(x,0)-u(x,0)) \, dx \geq 0 ,
\] 
as needed. 

Lastly, suppose $\rho<r<\infty$. If the minimum value $w(r,0)$ equals $0$ then there is nothing to prove. Suppose instead that $w(r,0)<0$. Since $w$ is not identically constant (remember $w(0,z)=0$), the strong minimum principle implies that $w>w(r,0)$ on the first quadrant (where $\deltastar w=0$). The Hopf lemma says that the normal derivative of $w$ at the boundary minimum point $(r,0)$ is positive; more precisely, since we have not shown $w$ is differentiable at the boundary, we apply a version of Hopf's lemma \cite[inequality (3.11)]{GT01} that gives the weaker conclusion  
\[
\liminf_{z \to 0+} \frac{w(r,z)-w(r,0)}{z} > 0 .
\]
We will prove this strict inequality to be impossible by showing 
\begin{equation} \label{eq:deriv0}
\limsup_{z \to 0+} \partial_z w(r,z) \leq 0 .
\end{equation}
Hence the situation $w(r,0)<0$ cannot occur.

Recall from Step 2 that $w(\infty,z)=0$ for all $z$ and that differentiation through the integral gives
\begin{align}
0 
& = \frac{d\ }{dz} w(\infty,z) \notag \\
& = \int_0^\infty \int_{\Sph^{n-1}} \left( \partial_z v(s\zeta,z) - \partial_z u(s\zeta,z) \right) d\zeta \, s^{n-1} \, ds \notag \\
& = \int_\Rn \partial_z v(x,z) \, dx -  \int_\Rn \partial_z u(x,z) \, dx , \label{eq:vminusu}
\end{align}
where the two integrands are integrable for large $|x|$ by the estimate \eqref{eq:Czr}. Hence for the value $r$, when $z>0$ we find 
\begin{align*}
& \partial_z w(r,z) \\
& = \int_{\B^n(r)} \partial_z v(x,z) \, dx - \int_{\B^n(r)} \partial_z u(x,z) \, dx \\
& = -\int_{\Rn \setminus \B^n(r)} \partial_z v(x,z) \, dx + \int_{\Rn \setminus \B^n(r)} \partial_z u(x,z) \, dx && \text{by \eqref{eq:vminusu}} \\
& \leq -\int_{\Rn \setminus \B^n(r)} \partial_z v(x,z) \, dx && \text{since $\partial_z u \leq 0$ by \eqref{eq:zderiv}} \\
& = b_n z \int_{\Rn \setminus \B^n(r)} \int_B \frac{1}{|(x,z)-(y,0)|^{n+1}} \, d\nu(y) dx && \text{by \eqref{eq:zderiv} applied to $B$} \\
& \leq b_n z \int_{\Rn \setminus \B^n(r)} \frac{1}{(|x|-\rho)^{n+1}} \, dx && \text{because $|y| \leq \rho < r \leq |x|$} \\
& \to 0
\end{align*}
as $z \to 0+$. This completes the proof of \eqref{eq:deriv0} and hence of Step 3. 

\subsection*{Step 4: $w=Jv-Ju \geq 0$  when $z \leq 0$} 

Step 3 shows $Jv-Ju \geq 0$ when $z \geq 0$. The same inequality holds when $z \leq 0$, because $B$ and $K$ lie in $\Rn$ and so $v$ and $u$ are even with respect to $z$, which means $Jv$ and $Ju$ are also even. Hence $Jv-Ju \geq 0$ for all $r \geq 0, z \in \R$, proving \autoref{pr:JvJuPos}.

\subsubsection*{Remark} Readers familiar with Baernstein's $\s$-function technique in complex analysis \cite[Chapter 11]{B19} will recognize that Step 3 above adapts his framework. The advantage here is that the potentials are harmonic in the upper halfspace $\{ z>0 \}$ and not merely subharmonic. That stronger information allows us to apply $J$ directly to $u$ rather than to its symmetric decreasing rearrangement on each slice, which would be the analogue of the standard $\s$-function approach.

\section{\bf Proof of \autoref{th:momentnone}: the moment inequalities}
\label{sec:momentnoneproof}

\subsection*{Proof of moment comparison} Let $\Phi(r)=r^q$ with $0<q \leq 2$, so that 
\[
\Psi(r) = \Phi^{\prime\prime}(r) + \frac{n}{r} \Phi^\prime(r) = q(q+n-1) r^{q-2} .
\]
Then $\Phi$ satisfies the hypotheses of \autoref{le:Phimoment} with the choice $\delta=q+n-1>0$. Additionally, $\Psi(R) > 0$ since $q>0$ and $\Psi^\prime(r) \leq 0$ since $q \leq 2$. By integrating over horizontal slices for $-r<z<r$, we find 
\begin{equation} \label{Jstuff}
\int_{\B^{n+1}(r)}(v(\xh)-u(\xh))\,d\xh=\int_{-r}^r J(v-u)(\sqrt{r^2-z^2},z) \,dz,
\end{equation}
and so formula \eqref{eq:momentJ} of the lemma and the inequality $Jv-Ju\geq0$ from \autoref{pr:JvJuPos} imply that 
\[
\int_K \Phi(|x|) \, d\mu \geq \int_B \Phi(|x|) \, d\nu .
\]  
Also, when $n \geq 2$ the function $\Phi(r)=\log r$ satisfies the hypotheses of \autoref{le:Phimoment} with $\delta=n-1>0$, and $\Psi(r) =(n-1) r^{-2}$ once more satisfies $\Psi(R) > 0$ and $\Psi^\prime(r) \leq 0$. Hence formula \eqref{eq:momentJ} of the lemma implies $\int_K \log |x| \, d\mu \geq \int_B \log |x| \, d\nu$. 

This logarithmic moment inequality remains true for $n=1$: it says simply that the logarithmic potential of $K$ at the origin is less than or equal to the potential of $B$ at the origin, which holds by the logarithmic case of \autoref{upperu} since $K$ and $B$ have the same logarithmic capacity; that is, 
\[
\int_K \log 1/|x| \, d\mu \leq V_{log}(K) = V_{log}(B) = \int_B \log 1/|x| \, d\nu .
\]
By definition, equality holds if and only if the origin is a regular point of $K$. 

\subsection*{Proof of other equality cases} $\Longleftarrow$ direction: suppose $K = B \cup Z$ for some set $Z$ of inner capacity zero. Since $K$ contains $B$ and the two sets have the same capacity, the equilibrium measures of $K$ and $B$ must be the same, $\mu=\nu$. Hence the moments on $K$ and $B$ are equal. 

 $\Longrightarrow$ direction: suppose the $\Phi$-moments of $K$ and $B$ agree, for some $\Phi$ that satisfies the hypotheses of \autoref{le:Phimoment} with $\Psi(R) > 0$ and $\Psi^\prime(r) \leq 0$. Expression \eqref{eq:momentJ} for the difference of moments equals zero, and so in particular the first term in \eqref{eq:momentJ} must equal zero. From \eqref{Jstuff} with $r=R$ we deduce $J(v-u) = 0$ on the semicircle of radius $R$ in the right half of the $rz$-plane. Applying the strong minimum principle for $\deltastar$ in the first quadrant (where $z>0$) implies that $J(v-u) \equiv 0$ in that quadrant, and similarly in the fourth quadrant (where $z<0$) and hence also by continuity on the positive $r$-axis (where $z=0$). 
 
At the radius $\rho$ of $B$, we get $Jv(\rho,0) = Ju(\rho,0)$. But on $B$ the inequality
\[
v=V_{n-1}(B) = V_{n-1}(K) \geq u
\]
holds (or instead with logarithmic energies when $n=1$), and since the integrals of $v$ and $u$ over $B$ agree, necessarily $v=u$ at almost every point of $B$ with respect to $n$-dimensional Lebesgue measure. Hence at each such point we know $u=V_{n-1}(K)$, which implies by \autoref{upperu} (with $m=n+1$) that the point lies in $K$, using here that the complement $\Rnp \setminus K$ is connected and unbounded because $K \subset \Rn$ (and so $\partial K = K$). Thus $K$ contains almost every point of $B$, and because $K$ is closed, it must contain all of $B$. Then since $K$ and $B$ have the same energy, the equilibrium measures of $K$ and $B$ must be the same, $\mu=\nu$, so that $u \equiv v$. Let $Z=K \setminus B$, so that $K=B \cup Z$. On $Z$, we have 
\[
u = v < V_{n-1}(B) = V_{n-1}(K) ,
\]
and so $Z$ has inner capacity zero by the equality statement in \autoref{upperu}. 

\subsection*{Alternative deduction of the logarithmic moment inequality}
Rather than using $\Phi(r)=\log r$ in the proof above, one may obtain the logarithmic moment inequality as a limiting case of $r^q$ as $q \to 0$, as follows. By rescaling the two sets, we may suppose $|x| \leq 1$ for all $x \in K \cup B$. We proved above that 
\[
\int_K |x|^q \, d\mu \geq \int_B |x|^q \, d\nu , \qquad q \in (0,2] . 
\]
Thus 
\[
\int_K \frac{1 - |x|^q}{q} \, d\mu \leq \int_B \frac{1 -  |x|^q}{q} \, d\nu .
\]
The integrand is nonnegative since $q>0$ and $|x| \leq 1$, and as $q \searrow 0$ it increases to $-\log |x|$, by convexity of $q \mapsto e^{q \log |x|}$. Hence  monotone convergence yields that 
\[
- \int_K \log |x| \, d\mu \leq - \int_B \log |x| \, d\nu ,
\]
which gives the desired inequality for the logarithmic moments.

\section*{Acknowledgments}
Richard Laugesen's research was supported by grants from the Simons Foundation (\#964018) and the National Science Foundation ({\#}2246537).

\appendix

\section{\bf Potential theoretic background}
\label{sec:background}

\subsection*{Physical interpretation} The Newtonian energy $V_1(K)$ in $3$-dimensions represents the least electrostatic energy (work) required to bring a unit of positive charges in from infinity and place them on a conductor of shape $K$. The charges then stay in place due to their mutual repulsion. A fuller discussion of this interpretation can be found in Baernstein \cite[Section 5.5]{B19}.  

Logarithmic energy in $2$ dimensions can be interpreted similarly by extending the charge distribution uniformly in the vertical direction and measuring the energy per unit length; see \cite[Section 1.5]{L93thesis}. 

\subsection*{Existence of an equilibrium measure} A weak-$*$ compactness argument (the Helly selection principle) applied to probability measures on $K$ shows the existence of a measure achieving the minimum in the definition of the energy, for both the logarithmic and Riesz situations with $0<p<n$ (see \cite[pp.{\,}131--132, 168]{L72}), and for more general kernels too; see \cite[Lemma 4.1.3]{BHS19}. 

\subsection*{Uniqueness of the equilibrium measure} A modern treatment of uniqueness for logarithmic equilibrium measure in all dimensions is given by Borodachov, Hardin and Saff \cite[Theorem 4.4.8]{BHS19}, and for Riesz equilibrium measure in \cite[Theorem 4.4.5]{BHS19}, assuming of course that $K$ has finite energy (positive capacity). These approaches handle also more general kernels. 

Uniqueness of the Riesz equilibrium measure can be found also in Landkof \cite[pp.{\,}132--133]{L72}. Uniqueness of the logarithmic equilibrium measure in $2$ dimensions is well known too: for example, \cite[p.\,168]{L72}, or the appealing variant in Saff and Totik \cite[Theorem I.1.3, Lemma I.1.8]{ST97}. 

\subsection*{Capacity formulas for the ball} The logarithmic capacity of an interval in $1$ dimension is a quarter of its length:
\[
\capzero(\overline{\B}^1) = \frac{1}{2}
\]
with equilibrium measure 
\[
d\mu(x) = \frac{1}{\pi} \, \frac{dx}{(1-x^2)^{1/2}} , \qquad -1<x<1 .
\]
This formula and related ones for $p$-capacity of an interval can be found in \cite[Proposition 4.6.1]{BHS19}. 

Assume next that $n \geq 2$. The unit ball has $(n-2)$-capacity  
\[
\capntwo(\overline{\B}^n) = 1 
\]
and the equilibrium measure is normalized $(n-1)$-dimensional surface area measure on $\Sph^{n-1}$. The Riesz $(n-1)$-capacity of the ball is 
\[
\capnone(\overline{\B}^n) =  \left( \frac{\Gamma(n/2)}{\Gamma(1/2)\Gamma((n+1)/2)} \right)^{\! \! 1/(n-1)} \! \! .
\]
 (For example, when the $2$-dimensional disk is regarded as a set in $3$ dimensions, it has Newtonian capacity $\capone(\overline{\B}^2) = 2/\pi$.) The equilibrium measure for $\capnone(\overline{\B}^n)$ is  
\[
d\mu(x) = \frac{2}{|\Sph^{n-1}|} \frac{1}{B \big( \frac{n}{2},\frac{1}{2} \big)} \, \frac{dx}{(1-|x|^2)^{1/2}} , \qquad |x|<1 ,
\]
where the beta function is $B( \frac{n}{2},\frac{1}{2}) = \Gamma(\frac{n}{2}) \Gamma(\frac{1}{2}) / \Gamma(\frac{n+1}{2})$. For all these claims, see \cite[Theorem 4.6.7, Proposition 4.6.4]{BHS19}, or else \cite[p.{\,}163, 172]{L72}, being mindful that when those authors define capacity they do not take the $p$-th root of the energy, whereas in this paper we do take the root.  
% Comment: 
% in Landkof's formulas, replace his dimension $p$ with $n$, and replace his exponent $\alpha$ with $n-p$. 

\subsection*{Properties of Newtonian and logarithmic equilibrium potentials} Consider a compact set $K \subset \Rm, m \geq 2$, with positive $(m-2)$-capacity $\capmtwo(K)>0$, and equilibrium measure $\mu$. The Newtonian equilibrium potential is
\[
u(x) = \int_K \frac{1}{|x-y|^{m-2}} \, d\mu(y) , \qquad x \in \Rm ,
\]
when $m \geq 3$, and when $m=2$ the logarithmic potential is 
\[
u(x) = \int_K \log \frac{1}{|x-y|} \, d\mu(y) , \qquad x \in \R^2 .
\]
The potential is a superharmonic function on $\Rm$ that is smooth and harmonic away from the support of $\mu$. The potential lies below the energy at every point:  
\begin{lemma} \label{upperu}
Let $m \geq 3$. For all $x \in \Rm$, 
\[
u(x) \leq V_{m-2}(K) .
\]
Equality holds at all interior points of $K$ and at all boundary points except perhaps on an exceptional subset $Z \subset \partial K$ of inner capacity zero. The inequality is strict, $u(x) < V_{m-2}(K)$, on the unbounded component of the complement $\Rm \setminus K$. 

If $K$ is a ball in $\Rm$ or $\R^{m-1}$, then $u(x) = V_{m-2}(K)$ on $K$ and $u(x) < V_{m-2}(K)$ on $\Rm \setminus K$, so that equality holds if and only if $x$ lies in the ball. 

When $m=2$, analogous inequalities hold for the logarithmic potential and energy. 
\end{lemma}
\begin{proof}
See Hayman and Kennedy \cite[Theorem 5.17]{HK76} for the inequality, with equality holding on $K$ except for a set of inner capacity zero; note that the exceptional set necessarily has Lebesgue measure zero by \cite[p.{\,}134]{L72}. Equality then holds at interior points by the super-meanvalue property of the potential. Strict inequality on the unbounded complement follows from the strong maximum principle. The ball properties can be found in Landkof \cite[pp.{\,}163--164]{L72}. 
\end{proof}
The exceptional set $Z \subset \partial K$ on which the potential is strictly less than the energy has Lebesgue measure zero in $\Rm$, and if $K$ happens to lie in $\R^{m-1}$ then $Z$ has zero $(m-1)$-dimensional Lebesgue measure too, by the same reasoning.  

The potential $u$ is continuous on $\Rm \setminus Z$: this continuity is clear for $x$ in the complement of the support of $\mu$, and so in particular $u$ is continuous on the complement of $K$, while continuity holds at each $x \in K \setminus Z$ because 
\[
V_{m-2}(K) = u(x) \leq \liminf_{y \to x} u(y) \leq \limsup_{y \to x} u(y) \leq V_{m-2}(K) 
\]
where the first inequality is by lower semicontinuity of the superharmonic function $u$ and the third inequality is immediate from \autoref{upperu}. When $m=2$, simply replace $V_{m-2}$ with $V_{log}$. In particular, if $K$ is a ball in $\Rm$ or $\R^{m-1}$ then $Z$ is empty by \autoref{upperu} and so the equilibrium potential is continuous on $\Rm$.

\bibliographystyle{plain}

\begin{thebibliography}{99}

\bibitem{ALT91}
A. Alvino, P.-L. Lions and G. Trombetti, 
\emph{Comparison results for elliptic and parabolic equations via symmetrization: a new approach}, 
Differential Integral Equations 4 (1991), 25--50.

\bibitem{B19}
A. Baernstein II, 
Symmetrization in Analysis, 
with David Drasin and Richard S. Laugesen. With a foreword by Walter Hayman. New Mathematical Monographs, 36. Cambridge University Press, Cambridge, 2019.

\bibitem{BLP11}
A. Baernstein II, R. S. Laugesen and I. E. Pritsker, 
\emph{Moment inequalities for equilibrium measures in the plane}, 
Pure Appl. Math. Q. 7 (2011), 51–86.

\bibitem{B04a}
D. Betsakos, 
\emph{Symmetrization, symmetric stable processes, and Riesz capacities}, 
Trans. Amer. Math. Soc. 356 (2004), 735--755.

\bibitem{B04b}
D. Betsakos, 
\emph{Addendum to: ``Symmetrization, symmetric stable processes, and Riesz capacities''}, Trans. Amer. Math. Soc. 356 (2004), 3821.

\bibitem{BHS19}
S. V. Borodachov, D. P. Hardin and E. B. Saff, 
Discrete Energy on Rectifiable Sets. Springer Monographs in Mathematics. Springer, New York, 2019.

\bibitem{GT01}
D. Gilbarg and N. S. Trudinger, 
Elliptic Partial Differential Equations of Second Order. Second edition, revised third printing. 
Springer--Verlag, Berlin, 2001. 

\bibitem{HK76}
W. K. Hayman and P. B. Kennedy, 
Subharmonic Functions. Vol. I. London Mathematical Society Monographs, No.{\,} 9. Academic Press (Harcourt Brace Jovanovich, Publishers), London--New York, 1976.

\bibitem{L72}
N. S. Landkof, 
Foundations of Modern Potential Theory. Translated from the Russian by A. P. Doohovskoy. Die Grundlehren der mathematischen Wissenschaften, Band 180. Springer--Verlag, New York--Heidelberg, 1972.

\bibitem{L93thesis}
R. Laugesen, 
\emph{Extremal problems involving logarithmic and Green capacity.} 
Thesis (Ph.D.) Washington University in St.\ Louis. 1993. 111 pp, ProQuest LLC.

\bibitem{L93}
R. Laugesen, 
\emph{Extremal problems involving logarithmic and Green capacity},  
Duke Math. J. 70 (1993), 445--480.

\bibitem{L22}
R. S. Laugesen, 
\emph{Minimizing capacity among linear images of rotationally invariant conductors}, 
Anal. Math. Phys. 12, 21 (2022).

\bibitem{M90}
P. Mattila. 
\emph{Orthogonal projections, Riesz capacities, and Minkowski content.} Indiana Univ. Math. J. 39 (1990), 185--198.

\bibitem{MH06}
P. J. M\'{e}ndez-Hern\'{a}ndez, 
\emph{An isoperimetric inequality for Riesz capacities},  
Rocky Mountain J. Math. 36 (2006), 675--682. 

\bibitem{PS51}
G. P\'{o}lya and G. Szeg\H{o}, 
Isoperimetric Inequalities in Mathematical Physics. 
Annals of Mathematics Studies, No.\ 27, Princeton University Press, Princeton, N.J., 1951.

\bibitem{P23}
I. Pritsker, 
\emph{Mahler measure of polynomial iterates}, 
Canad. Math. Bull. 66 (2023), 881--885.

\bibitem{ST97}
E. B. Saff and V. Totik, 
Logarithmic Potentials with External Fields. 
Appendix B by Thomas Bloom. Grundlehren der Mathematischen Wissenschaften, 316. Springer--Verlag, Berlin, 1997.

\bibitem{SZ04}
A. Yu. Solynin and V. Zalgaller. 
\emph{An isoperimetric inequality for logarithmic capacity of polygons.} Ann. of Math. (2) 159 (2004), 277--303.

\bibitem{T16}
G. Talenti, 
\emph{The art of rearranging}, 
Milan J. Math. 84 (2016), 105--157.

\bibitem{W83}
T. Watanabe, 
\emph{The isoperimetric inequality for isotropic unimodal L\'{e}vy processes},  
Z. Wahrsch. Verw. Gebiete 63 (1983), 487--499. 

\end{thebibliography}

\end{document}